\documentclass[runningheads]{llncs}
\usepackage[T1]{fontenc}
\usepackage{graphicx}
\usepackage[pdftex,colorlinks,linkcolor=black,urlcolor=black,citecolor=black]{hyperref}
\usepackage{tikz}
\ifpdf
\fi

\usepackage{amsthm}

\newcounter{myclaim}
\newenvironment{myclaim}[1][]%
	{\refstepcounter{myclaim}\vspace{1ex}\noindent {(\it\arabic{myclaim}) {#1}{}}\it}{\vspace{1ex}}
\newenvironment{proofclaim}[1][]%
	{\noindent {}{#1}{}}{ This proves~(\arabic{myclaim}).\vspace{1ex}}

\usepackage{amsmath}
\definecolor{ForestGreen}{RGB}{34,139,34}

\newcommand {\sm} {\setminus}

\begin{document}
\title{A structural description of Zykov and Blanche Descartes graphs}
%
%
\author{Malory Marin \and
Stéphan Thomassé \and
Nicolas Trotignon \and Rémi Watrigant}
\authorrunning{ }
%
\institute{\small ENS de Lyon, CNRS, Université Claude Bernard Lyon 1, LIP, UMR 5668, \\69342 Lyon Cedex 07, France}
%
\maketitle              

\begin{abstract} 
In 1949, Zykov proposed the first explicit construction of triangle-free graphs with arbitrarily large chromatic number. We define a Zykov graph as any induced subgraph of a graph created using Zykov's construction. We give a structural characterization of Zykov graphs based on a specific type of stable set, that we call splitting stable set. It implies that recognizing  this class is NP-complete, while being FPT in the treewidth of the input graph.  We provide similar results for the Blanche Descartes construction. 
\end{abstract}

\section{Introduction}
A class of graphs is \emph{hereditary} if it is closed under taking induced subgraphs.  A hereditary class of graphs $\mathcal{C}$ is \emph{$\chi$-bounded} if there exists a function $f$ such that every graph $G$ in $\mathcal C$ satisfies $\chi(G)\leq f(\omega(G))$. Some classes are $\chi$-bounded and some are not.  There has been much research about this topic. Some questions remain open, such as the Gy\'arf\'as-Sumner conjecture, see~\cite{scottSeymourSurvey}.

It turns out that some conjectures about $\chi$-boundedness have been disproved by studying carefully some constructions of triangle-graphs of arbitrarily large chromatic number defined much before the conjectures were first stated.  In particular, the construction of Burling~\cite{burling1965coloring}, defined in 1965, turned out to be a counter-example to a conjecture of Scott (see~\cite{DBLP:journals/dcg/PawlikKKLMTW13}), to a conjecture of Trotignon (see~\cite{davies:wheels,PT-Burling1,PT-Burling2,PT-Burling3}) and to a conjecture of Thomassé, Trotignon and Vu\v skovi\'c (see~\cite{rzkazewski2024polynomial}).  Also the question of solving several algorithmic problems in non-$\chi$ bounded classes attracted some attention (see~\cite{rzkazewski2024polynomial}).

More insight about other classical constructions giving non-$\chi$-bounded classes is therefore of interest. Here, we study the two first constructions of triangle-free graphs of arbitrarily large chromatic number, defined respectively by Zykov~\cite{Zyk49} in 1949 and Blanche Descartes~\cite{descartes1947three,descartes1954solution} in 1954. We view them as hereditary classes, defined by taking all induced subgraphs of graphs occuring in the constructions. Our main result is a structural characterization of both Zykov and Blanche Descartes graphs (see Theorems~\ref{th:z} and~\ref{th:BD}). They imply that recognizing both classes is NP-complete (see Theorems~\ref{th:npc} and~\ref{th:BDnpc}), but tractable in polynomial time when parameterized by the treewidth of the input graph (see Theorems~\ref{thm:ZykovFPT} and~\ref{thm:BDFPT}). They also imply that some classical optimization problems are NP-hard for Zykov of Blanche Descartes graphs (see Theorems~\ref{thm:BDMIS} and \ref{thm:BDColoring}).  We should stress that our motivation is not really algorithmic. We rather view our results as meaningful regarding the structure of Zykov and Blanche Descartes graphs. For instance, we observe that every Blanche Descartes graph is a Zykov graph (see Lemma~\ref{l:ZiBD}).  Our results also allow easy constructions of non-Zykov and non-Blanche-Descartes graphs, which is not obvious. In particular we show that there exists non-Zykov graphs of arbitrarily large girth (see Theorem~\ref{thm:ZykovLargeGirth}). However, NP-completeness makes unlikely the existence of a nice description of minimal non-Zykov and minimal non-Blanche-Descartes graphs.

Before giving the formal definitions, let us briefly survey some other famous constructions of triangle-free graphs of high chromatic number. For the famous Mycielski's construction~\cite{mycielski1955coloriage}, the situation is simple: an easy induction shows that any triangle-free graph is an induced subgraph of some Mycielski graph. Thus, the hereditary class of induced subgraphs of Mycielski graphs is exactly the class of triangle-free graphs, which is obviously recognizable in polynomial time. For Burling graphs, several characterizations are given by Pournajafi and Trotignon~\cite{PT-Burling1}, and  recently, Rz{\k{a}}{\.z}ewski and Walczak~\cite{rzkazewski2024polynomial} proved that recognizing Burling graphs and computing a maximum stable set in a Burling graph are polynomial time tractable. The structure of Blanche Descartes graphs has been already investigated by Kostochka and Ne\v set\v ril~\cite{DBLP:journals/cpc/KostochkaN99}, but their results are of a completely different flavour. 

Let us now give all the formal definitions. 

\paragraph{Zykov graphs.} Zykov's construction is inductive. The graph $Z_1$ is the graph on one vertex. Suppose that $k\geq 1$ is an integer and that $Z_1$, \dots, $Z_k$ are all defined. Then $Z_{k+1}$ is obtained as follows:

\begin{itemize}
\item Take the disjoint union of $Z_1$, \dots, $Z_k$.
\item For each $k$-tuple of vertices $v_1\in V(Z_1)$, \dots, $v_k \in
  V(Z_k)$, add a vertex $v$ adjacent to $v_1$, \dots, $v_k$. 
\end{itemize}

For the proofs in the next sections, we denote by $S_{k+1}$ the set of
vertices added in the second item of the definition above.

For the sake of completeness, let us recall the main properties of
Zykov's construction. It is a routine matter to prove by induction
that for all $k\geq 1$, $Z_k$ is triangle-free and $k$-colorable. To
prove that at least $k$ colors are needed to color $Z_{k}$, the proof
is also by induction.  Suppose it is true for 1, \dots, $k$ and
suppose for a contradiction that $Z_{k+1}$ is $k$-colorable.  In a
$k$-coloring of $Z_{k+1}$, color 1 must be used in $Z_1$, color 2 in
$Z_2$, and so on (up to a permutation of colors).  Hence, there exists
a $k$-tuple of vertices $v_1\in V(Z_1)$, \dots, $v_k \in V(Z_k)$ that
uses $k$ distinct colors. One more color is needed for the vertex
adjacent to this $k$-tuple.  We proved that $Z_k$ is a triangle-free
graph of chromatic number $k$. 

A \emph{Zykov graph} is any graph isomorphic to an induced subgraph of
$Z_k$ for some integer $k \geq 1$. In Section~\ref{sec:z}, we prove our results about Zykov graphs. 

\paragraph{Blanche Descartes graphs.} Blanche Descartes's construction is also inductive. The graph $D_1$ is the graph on one vertex. Suppose that $k\geq 1$ is an integer and that $D_k$ is defined and has $n$ vertices.  Then $D_{k+1}$ is obtained as follows:

\begin{itemize}
\item Take a stable set $S$ on $k(n-1)+1$ new vertices.
\item For each $n$-tuple of vertices $T$ in $S$, add a copy of $D_k$ and a matching between $T$ and $D_k$. 
\end{itemize}

Let us recall the main properties of Blanche Descartes's construction. It is a routine matter to prove by induction that for all $k\geq 1$, $D_k$ is $(C_3, C_4, C_5)$-free and $k$-colorable. To prove that at least $k$ colors are needed to color $D_{k}$, the proof is also by induction. Suppose it is true for $1, \dots, k$ and suppose for a contradiction that $D_{k+1}$ is $k$-colorable.  By the pigeonhole principle, in a $k$-coloring of $D_{k+1}$, at least one $n$-tuple of vertices of $S$ must be monochromatic. Hence, the copy of $D_k$ that is matched to this $n$-tuple is colored with $k-1$ colors, a contradiction to the induction hypothesis. We proved that $D_k$ has chromatic number $k$.

Observe that as presented here, the construction is not deterministic because different matchings can lead to non-isomorphic graphs. More formally, $D_k$ should be defined as the set of graphs which can be constructed using this method after $k$ steps. To avoid heavy notation, along the paper we consider that $D_k$ is one of the Blanche Descartes graph that can be constructed after $k$ steps. We leave as an open question whether the characterization of Blanche Descartes graphs that we obtain is independent of the chosen matchings. 

A \emph{Blanche Descartes graph} is any graph isomorphic to an induced subgraph of any $D_k$ for some integer $k \geq 1$. In Section~\ref{sec:BD}, we prove our results about Blanche Descartes graphs.

\section{Zykov graphs}\label{sec:z}

In this section, we start by providing a characterization of Zykov graphs based on certain special stable sets. Using this characterization, we derive the main result: the proof that recognizing Zykov graphs is NP-complete.

\subsection{Characterization of Zykov graphs}

A stable set $A$ in some graph $G$ is \emph{splitting} if every vertex
of $A$ has at most one neighbor in each connected component of
$G\sm A$. Note that every graph contains a splitting stable set: the
empty set. Note also that $G$ contains a non-empty splitting stable
set if and only if at least one of its connected component does.

\begin{lemma}
  \label{lemma:z}
  For all graphs $G$, the following conditions are equivalent:

  \begin{enumerate}
  \item\label{i:z} $G$ is a Zykov graph.
   \item\label{i:as} Every induced subgraph $H$ of $G$ contains a
    non-empty splitting stable set $A$ such that $H \setminus A$ is a Zykov
    graph.
  \item\label{i:es} $G$ contains a non-empty splitting stable set $A$ such that
    $G \setminus A$ is a Zykov graph.
  \end{enumerate}
\end{lemma}

\begin{proof}
  To prove that~\eqref{i:z} implies~\eqref{i:as}, it is enough to
  prove that every connected induced subgraph $H$ of some graph $Z_k$
  contains a non-empty splitting stable set (its removal yields a
  Zykov graph because Zykov graphs form a hereditary class by
  definition).  So, let $H$ be a connected subgraph of some graph
  $Z_k$ and suppose that $k$ is minimal with respect to this property.
  By the minimality of $k$, $H$ contains some vertices of the set
  $S_k$ used to construct $Z_k$ from $Z_1$, \dots, $Z_{k-1}$ as in the
  definition.  Hence, $V(H)\cap S_k$ is a non-empty splitting stable
  set of $H$, so Condition~\eqref{i:as} holds.

  Clearly, \eqref{i:as} implies~\eqref{i:es}.

  To prove that~\eqref{i:es} implies~\eqref{i:z}, suppose that $G$
  contains a non-empty splitting stable set $A=\{a_1, \dots, a_n\}$
  such that $G\sm A$ is a Zykov graph.  Our goal is to find a large
  enough integer $m$ such that $G$ can be viewed as an induced
  subgraph of $Z_m$.

  Let $C_1$, \dots, $C_\ell$ be the connected components of $G\sm A$.
  Let $k$ be an integer large enough to satisfy the following
  properties: $|V(Z_k)|\geq |A|$ and each component of $G\sm A$ is an
  induced subgraph of $Z_{k}$ (this is possible since $G\sm A$ is a
  Zykov graph by assumption).  We claim that $G$ is isomorphic to an
  induced subgraph of $Z_{k+\ell+1}$.  Recall that $Z_{k+\ell+1}$ is
  obtained from the disjoint union of $Z_1$, \dots, $Z_{k+\ell}$.
  Since $|V(Z_k)|\geq |A|$, we may consider distinct vertices
  $b_1, \dots, b_n$ in $Z_k$.  We view each component $C_i$ as an
  induced subgraph of $Z_{k+i}$.  For each vertex $a_i$, we define a
  $(k+\ell)$-tuple $T_i = (c_{i, 1}, \dots, c_{i, k+\ell})$ of
  vertices with $c_{i, j}\in V(Z_j)$ for $j=1, \dots, k+\ell$.  For
  $j=1, \dots, k-1$, we choose $c_{i, j}$ to be an arbitraty vertex in
  $Z_j$. For $j=k$, we choose $c_{i, j}=b_i$.  For
  $j=k+1, \dots, k+\ell$, if $a_i$ has a neighbor $a_{i, j}$ in $C_j$,
  we set $c_{i, j} = a'_{i, j}$ and if $a_i$ has no neighbor in $C_j$,
  we choose $c_{i, j}$ to be any vertex from $Z_j\sm C_j$ (note that
  $Z_j\sm C_j\neq \emptyset$ since $j>k$). We now view $a_i$ as the
  vertex associated to the $(k+\ell)$-tuple $T_i$.  Note that the
  vertices $c_{i, k}$ for $i=1, \dots, n$ guaranty that all $T_i$'s
  are distinct. This proves that $G$ is an induced subgraph of
  $Z_{k+\ell + 1}$.
\end{proof}

\begin{theorem}
  \label{th:z}
  A graph $G$ is a Zykov graph if and only if all induced subgraphs
  of $G$ contain a non-empty splitting stable set.
\end{theorem}

\begin{proof}
  If $G$ is a Zykov graphs, the conclusion holds by
  Condition~\eqref{i:as} of Lemma~\ref{th:z}. Conversely, suppose that
  $G$ is such that all induced subgraphs of $G$ contain a non-empty
  splitting stable set.  Let us prove by induction on $|V(G)|$ that
  $G$ is a Zykov graph. If $|V(G)|=1$, $G=Z_1$ is a Zykov graph. If
  $|V(G)|>1$, then by assumption $G$ contains a non-empty splitting
  stable set $A$.  By the induction hypothesis, $G\sm A$ is a Zykov
  graph.  Hence, by Condition~\eqref{i:es} of Lemma~\ref{lemma:z}, $G$ 
  is a Zykov graph.
\end{proof}

\paragraph{Some properties of Zykov graphs.} We state now several simple observations, needed in the sequel or possibly useful for a further study of the structure of Zykov graphs.

\begin{lemma}
  \label{l:subdiv}
  The class of Zykov graphs is closed under subdividing edges.
\end{lemma}

\begin{proof}
  Let $H$ be obtained from some Zykov graph $G$ by subdividing
  edges. For every $X\subseteq V(G)$, if $A$ is splitting stable set
  of $G[X]$, then it is a splitting stable set of $H[X]$.  It follows
  from Theorem~\ref{th:z} that all induced subgraphs of $H$ have a
  non-empty splitting stable set. Hence, by Theorem~\ref{th:z} again,
  $H$ is a Zykov graph.
\end{proof}

\begin{lemma}
  The class of Zykov graphs is closed under taking subgraphs.
\end{lemma}

\begin{proof}
  This follows from Lemma~\ref{l:subdiv} and the fact that a subgraph
  of some graph $G$ can be obtained from $G$ by a sequence of
  subdivision of edges et delition of vertices.
\end{proof}

We now give tools to prove that some graphs are not Zykov graphs.

\begin{lemma}
  \label{l:2c}
  If $C$ is a cycle of some graph $G$, then every splitting stable set
  $A$ of $G$ contains either no vertex of $C$, or at least two
  vertices of $C$.
\end{lemma}

\begin{proof}
  If $A$ contains a unique vertex $v$ of $C$, then $v$ has at least
  two neighbors in the connected component of $G\sm A$ that contains
  $C\sm \{v\}$, a contradiction to the definition of a splitting
  stable set.
\end{proof}

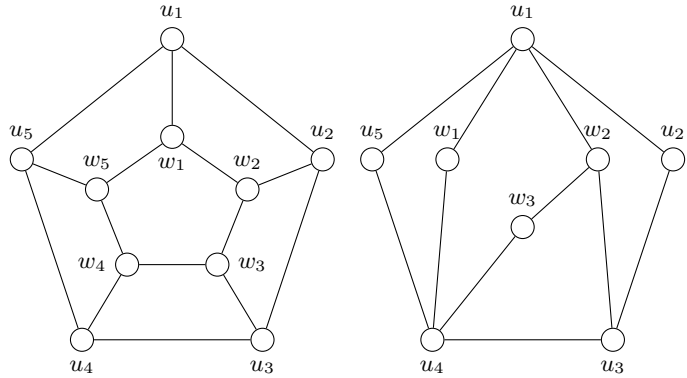
\begin{figure}
\centering
\begin{tikzpicture}
\node[draw, circle] (1) at (0,1.7) [label=below:$w_1$]{};
\node[draw, circle] (2) at (1,1) [label=$w_2$]{};
\node[draw, circle] (3) at (0.6,0) [label=right:$w_3$]{};
\node[draw, circle] (4) at (-0.6,0) [label=left:$w_4$]{};
\node[draw, circle] (5) at (-1,1) [label=$w_5$]{};

\draw (1) -- (2) ;
\draw (2) -- (3) ;
\draw (3) -- (4) ;
\draw (4) -- (5) ;
\draw (5) -- (1) ;

\node[draw, circle] (6) at (0,3) [label=$u_1$]{};
\node[draw, circle] (7) at (2,1.4) [label=$u_2$]{};
\node[draw, circle] (8) at (1.2,-1) [label=below:$u_3$]{};
\node[draw, circle] (9) at (-1.2,-1) [label=below:$u_4$]{};
\node[draw, circle] (10) at (-2,1.4) [label=$u_5$]{};

\draw (6) -- (7) ;
\draw (7) -- (8) ;
\draw (8) -- (9) ;
\draw (9) -- (10) ;
\draw (10) -- (6) ;

\draw (1) -- (6) ;
\draw (2) -- (7) ;
\draw (3) -- (8) ;
\draw (4) -- (9) ;
\draw (5) -- (10) ;

\end{tikzpicture}
\begin{tikzpicture}

\node[draw, circle] (1) at (-1,1.4) [label = $w_1$]{};
\node[draw, circle] (2) at (1,1.4) [label = $w_2$]{};
\node[draw, circle] (3) at (0,0.5) [label = $w_3$]{};

\node[draw, circle] (6) at (0,3) [label=$u_1$]{};
\node[draw, circle] (7) at (2,1.4) [label=$u_2$]{};
\node[draw, circle] (8) at (1.2,-1) [label=below:$u_3$]{};
\node[draw, circle] (9) at (-1.2,-1) [label=below:$u_4$]{};
\node[draw, circle] (10) at (-2,1.4) [label=$u_5$]{};

\draw (6) -- (7) ;
\draw (7) -- (8) ;
\draw (8) -- (9) ;
\draw (9) -- (10) ;
\draw (10) -- (6) ;
\draw (6) -- (1) ;
\draw (6) -- (2) ;
\draw (2) -- (3) ;
\draw (1) -- (9) ;
\draw (3) -- (9) ;
\draw (2) -- (8) ;

\end{tikzpicture}
   
  \caption{Graphs $F$ and $F'$\label{f:nz}}
\end{figure}

Lemma~\ref{l:2c} gives a simple criterion to prove that some graphs are
not Zykov. The graph $F$ represented in Figure~\ref{f:nz} is not
Zykov.  Indeed, if it were Zykov, it would contain by
Theorem~\ref{th:z} a non-empty splitting stable set $A$.  Up to
symmetry, $A$ must contain $u_1$, and therefore by a sequence of
applications of Lemma~\ref{l:2c}, $A$ must contain $w_2$, $u_3$, $w_4$ and
$u_5$, a contradiction.  By a similar argument one can easily check
that the graph $F'$ from the same figure is not Zykov.  It is also
easy to check that $F'$ is the smallest non-Zykov triangle-free graph.
Indeed, by Lemma~\ref{l:bip}, we know that such a graph must be
non-bipartite, and must therefore contain an odd cycle of length at
least~5. A brute-force check of all possible ways to attach two
vertices to a cycle of length~5 of to add chords to a cycle of
length~7 only yields graphs with a splitting stable set whose removal
yields a forest, so all triangle-free graphs on at most seven vertices
are Zykov graphs.

We state the next observation in a lemma for reference in the next section.

\begin{lemma}
  \label{l:h}
  If a graph $G$ contains the graph $H$ represented on
  Figure~\ref{f:npc-H} as an induced subgraph, then for every
  splitting stable set $A$ of $G$, $A\cap V(H) = \emptyset$ or
  $A\cap V(H) = \{a, a', a'', a'''\}$.
\end{lemma}

\begin{proof}
  Clear by several applications of Lemma~\ref{l:2c}. 
\end{proof}

\begin{figure}
\centering
\begin{tikzpicture}

\node[draw,circle, fill =red] (0) at (0,0) [label = left:$v$]{};
\node[draw,circle, fill =red] (1) at (2,0) [label = above:$v'$]{};
\node[draw,circle, fill =red] (2) at (4,0) [label = right:$v''$]{};
\node[draw,circle, fill =ForestGreen] (3) at (1,0.8) [label = above:$a$]{};
\node[draw,circle, fill =ForestGreen] (4) at (1,-0.8) [label = above:$a'$]{};
\node[draw,circle, fill =ForestGreen] (5) at (3,0.8) [label = above:$a''$]{};
\node[draw,circle, fill =ForestGreen] (6) at (3,-0.8) [label = above:$a'''$]{};

\draw (0) -- (3) ;
\draw (0) -- (4) ;
\draw (3) -- (1) ;
\draw (4) -- (1) ;
\draw (1) -- (5) ;
\draw (1) -- (6) ;
\draw (5) -- (2) ;
\draw (6) -- (2) ;

\draw (0) to[out= 90, in = 180] (2,2) to[out= 0, in = 90] (2) ;
\end{tikzpicture}

  \caption{Graph $H$\label{f:npc-H}}
\end{figure}


\paragraph{On specific graph classes.} The well-known Gy\'arf\'as-Sumner conjecture states that any hereditary class of graphs that forbids a tree is $\chi$-bounded. It is already established that Zykov graphs are not counterexamples to this conjecture, implying that every tree appears in the Zykov construction. Theorem~\ref{th:z} provides a direct proof of this fact, and we can further strengthen the result by showing that any bipartite graph is a Zykov graph.

\begin{lemma}
  \label{l:fz}
  All forests are Zykov graphs.
\end{lemma}

\begin{proof}
  Forests are Zykov graphs by Theorem~\ref{th:z} because a vertex of
  degree at most~1 in any graph forms a splitting stable set.
\end{proof}

\begin{lemma}
  \label{l:bip}
  All bipartite graphs are Zykov graphs. 
\end{lemma}

\begin{proof}
  Any side of the bipartition of a bipartite graph forms a splitting stable
  set whose removal yields a forest. Hence, the conclusion holds by
  Condition~\eqref{i:es} of Lemma~\ref{lemma:z} and Lemma~\ref{l:fz}.
\end{proof}

\begin{lemma}
  \label{l:allSub}
  Let $G$ be any graph (possibly not Zykov). If $H$ is obtained by
  subdividing all edges of $G$ at least once, then $H$ is a Zykov
  graph.
\end{lemma}

\begin{proof}
  If each edge of $G$ is subdivided exactly once, then $H$ is
  bipartite, so the result follows from Lemma~\ref{l:bip}.  If there
  are more subdivisions, the result follows from Lemma~\ref{l:subdiv}. 
\end{proof}

While all trees are Zykov graphs, one may wonder what happens when a graph resembles a tree, or more formally, when it has bounded treewidth. We first provide a brief definition of treewidth for completeness. A tree-decomposition of a graph $G$ is a pair $(T, \mu)$ where $T$ is a tree and $\mu$ is a map from $V(T)$ to $2^{V(G)}$ such that:
\begin{itemize}
    \item for every $v \in V(G)$, the set $\{t \in V(T) : v \in \mu(t)\}$ induces a non-empty subtree of $T$, and
    \item for every $uv \in E(G)$, there exists $t \in V(T)$ such that $\{u, v\} \subseteq \mu(t)$.
\end{itemize}
The width of $(T, \mu)$ is $\max_{t\in V(T)} |\mu(t)|-1$, and the treewidth of $G$, denoted by $\text{tw}(G)$, is the minimum width of a tree-decomposition of $G$. Moreover, there exists an algorithm that returns a tree-decomposition of width $O(\text{tw}(G))$ in time $2^{O(\text{tw}(G))} \cdot n$. For more details, see Cygan et al.~\cite{cygan2015parameterized}.

A straightforward generalization of the graph $F$ (illustrated in Figure~\ref{f:nz}) with arbitrarily large odd cycles provides an example of a non-Zykov graph that is triangle-free and has treewidth 3. 
At first glance, it seems non-trivial to decide whether a graph is a Zykov graph, even if it has bounded treewidth. However, using Theorem~\ref{th:z}, we show that there exists a linear-time algorithm to determine whether the input graph is a Zykov graph when it has bounded \emph{treewidth}.

A Fixed Parameter Tractable (FPT) algorithm is an algorithm that decides whether an instance of a decision problem is positive in time $f(k) \cdot n^c$, where $c$ is some fixed constant, $f$ is a computable function, $n$ is the size of the instance, and $k$ is a chosen parameter depending on the problem.

\begin{theorem}\label{thm:ZykovFPT}
    There exists an algorithm that, given an input graph $G$, decides if $G$ is a Zykov graph in time $f(\text{tw}(G)) \cdot n$, where $n$ is the number of vertices of $G$, $tw(G)$ its treewidth and $f$ a computable function.
\end{theorem}

\begin{proof}
    By a classical theorem of Courcelle, it suffices to construct a monadic second-order logic formula, denoted by $\Phi$, such that a graph $G$ satisfies $\Phi$ if and only if it is a Zykov graph. We provide a high-level overview of the formula~$\Phi$, based on the characterization from Theorem~\ref{th:z}.
    
    A graph $G = (V, E)$ is a Zykov graph if and only if, for any subset of vertices $V' \subseteq V$, there exists a subset $S \subseteq V'$ such that $S$ is an splitting stable set for the induced subgraph $G[V']$. The property of being a splitting set can be expressed as follows: for any subset $C \subseteq V' \setminus S$, if $G[C]$ is connected, then each vertex $v \in S$ is adjacent to at most one vertex in~$C$. Furthermore, the properties of a subset of vertices inducing either a stable set or a connected subgraph can also be expressed in monadic second-order logic.
\end{proof}

Notice that in Figure~\ref{f:nz}, the graph $F'$ has treewidth $3$ and is not a Zykov graph. In Figure~\ref{f:nztw2}, we give a non-Zykov triangle-free graph of treewidth $2$, which is made of two copies of the graph $H$ from Lemma~\ref{l:h}. Note that for the edge shared by both copies of $H$, exactly one must belong to every stable splitting set, while the other must not. However, their roles are reversed in each copy. Consequently, whenever a stable splitting set is non-empty, it inevitably leads to a contradiction at one of these two vertices.
\begin{figure}
    \centering
    \begin{tikzpicture}
    
    \foreach \i/\j in {0/2, 1/1,2/1,3/0}{
        \draw (\i,\j) -- (\i+1,\j);
        \draw (\i,\j) -- (\i,\j-1);

    }
    
    \foreach \i/\j in {1/1,3/0, 4/-1}{
        \draw (\i,\j) -- (\i-1,\j);
        \draw (\i,\j) -- (\i,\j+1);
    }
    \draw (1,0) -- (2,0);
    
    \draw (0,2)  to[bend right = 90, looseness=2] (2,0);
    \draw (2,1)  to[bend left = 90, looseness=2] (4,-1);

    \node[draw, circle, fill=white] () at (0,2) {};
    \node[draw, circle, fill=white] () at (1,2) {};

    \node[draw, circle, fill=white] () at (3,-1) {};
    \node[draw, circle, fill=white] () at (4,-1) {};

    \foreach \i in {0,1,2,3}{
        \node[draw, circle, fill=white] () at (\i,1) {};
        \node[draw, circle, fill=white] () at (\i+1,0) {};

    }
    \end{tikzpicture}
    \caption{Non-Zykov graph of treewidth $2$}
    \label{f:nztw2}
\end{figure}
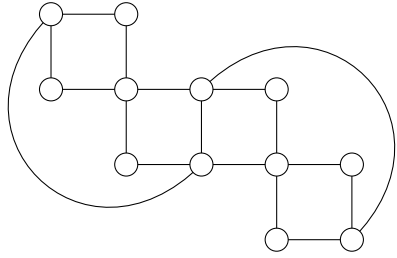

While we have established that recognizing Zykov graphs is tractable for graphs with bounded treewidth, the situation changes drastically for arbitrary graphs. In the following, we prove that the problem becomes NP-complete when the restriction on treewidth is lifted.

\subsection{NP-hardness}
\label{sec:npc}

\begin{theorem}\label{th:npc}
    Recognizing Zykov graphs is NP-complete.
\end{theorem}
\begin{proof}
First observe that the problem of recognizing Zykov graphs is in
NP. To see this, observe that a graph $G$ is a Zykov graph if and only
if it can be partitioned into $n$ stable sets $A_1$, \dots, $A_n$ in
such a way that for all $k=1, \dots, n-1$, $A_k$ is a splitting stable
set of $G[A_{k+1} \cup \dots \cup A_n]$.  This follows from
Condition~\eqref{i:es} of Lemma~\ref{lemma:z} by an easy induction.
Since deciding whether a given set is a splitting stable set is easily
performed in polynomial time, the stable sets $A_1$, \dots, $A_n$ certify
that a Zykov graph is one.

We now reduce 3-{\sc sat} to the problem of recognizing Zykov graphs.
Consider an instance $\mathcal I$ of 3-{\sc sat} made of $n$ variables
$x_1$, \dots, $x_n$ and $m$ clauses $C_1, \dots, C_m$ each on three
variables.  For each $j=1, \dots, m$, 
$C_j=y_{j, 1} \vee y_{j, 2} \vee y_{j, 3}$ where for all $k=1, 2, 3$,
there exists $1\leq i \leq n$ such that $y_{j, k} = x_i$ or
$y_{j, k} = \overline{x_i}$.

We define a graph $G_{\mathcal I}$ depending on $\mathcal I$.  To make
the explanations easier, a color (either red or green) is assigned to
some vertices of $G_{\mathcal I}$.  The red vertices will turn out to
be in no splitting stable set of $G_{\mathcal I}$ and the green
vertices in all non-empty splitting stable set of $G_{\mathcal
  I}$. The presence of uncolored vertices in some non-empty splitting
stable set will depend on $\mathcal I$.

Prepare a copy of the graph $H$ with vertex-set $\{v, v', v'', a, a',
a'', a'''\}$ as represented on Figure~\ref{f:npc-H}.  Give color red to
$v$, $v'$ and $v''$.  Give color green to $a, a', a''$ and $a'''$. 
This graph $H$ will allow us to force many other vertices of $G_{\mathcal{I}}$ to be part of a splitting stable set or not. More precisely, any other red vertex $u$ of $G_{\mathcal{I}}$ will be made adjacent to $\{a, a', a'', a'''\}$, so that $(V(H) \setminus \{v'\}) \cup \{u\}$ will induce a graph isomorphic to $H$, and any other green vertex $u$ of $G_{\mathcal{I}}$ will be made adjacent to $\{v, v'\}$, so that $(V(H) \setminus \{a\})\cup \{u\}$ will induce a graph isomorphic to $H$.

Now, for each variable $x_i$, prepare a graph $G_{i}$ on 5 vertices
$t_i$, $f_i$, $b_{i, 1}$, $b_{i, 2}$ and $b_{i, 3}$.  Add edges in
such way that $t_ib_{i,1}b_{i,2}b_{i,3}f_it_i$ is a cycle. Give color
green to vertex $b_{i, 2}$.  Give color red to vertices $b_{i, 1}$ and
$b_{i, 3}$.  Observe that no color is given to vertices $t_i$ and
$f_i$. 

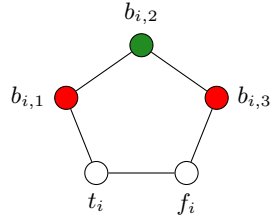
\begin{figure}
\centering 
  
\begin{tikzpicture}
\node[draw, circle, fill = ForestGreen] (1) at (0,1.7) [label=above:$b_{i,2}$]{};
\node[draw, circle, fill = red] (2) at (1,1) [label=right:$b_{i,3}$]{};
\node[draw, circle] (3) at (0.6,0) [label=below:$f_i$]{};
\node[draw, circle] (4) at (-0.6,0) [label=below:$t_i$]{};
\node[draw, circle, fill = red] (5) at (-1,1) [label=left:$b_{i,1}$]{};

\draw (1) -- (2) ;
\draw (2) -- (3) ;
\draw (3) -- (4) ;
\draw (4) -- (5) ;
\draw (5) -- (1) ;
\end{tikzpicture}
\caption{Graph $G_i$\label{f:npc-Gi}}
\end{figure}

For each clause $C_j$, prepare a graph $G_{C_j}$ on 21 vertices
$c_{j, 1}$, \dots $c_{j, 6}$ and $d_{j, k, \ell}$ with
$k\in \{1, 2, 3\}$ and $\ell \in \{1, \dots, 5\}$. Add the edges to create three cycles of 5 vertices connected by one edge to a cycle on 6 vertices, as illustrated in Figure~\ref{f:npc-GCj}. Four vertices of $G_{C_j}$ are colored in green, and eight of them in red.

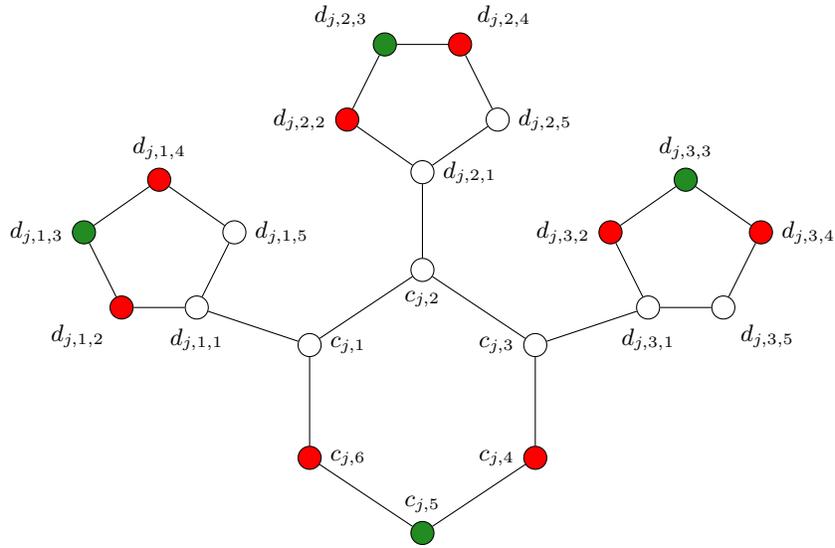
\begin{figure}
\centering

\begin{tikzpicture}
\node[draw, circle] (1) at (0,2.5) [label=below:$c_{j,2}$]{};
\node[draw, circle] (2) at (1.5,1.5) [label=left:$c_{j,3}$]{};
\node[draw, circle, fill = red] (3) at (1.5,0) [label=left:$c_{j,4}$]{};
\node[draw, circle, fill = ForestGreen] (4) at (0,-1) [label=above:$c_{j,5}$]{};
\node[draw, circle, fill = red] (5) at (-1.5,0) [label=right:$c_{j,6}$]{};
\node[draw, circle] (6) at (-1.5,1.5) [label=right:$c_{j,1}$]{};

\draw (1) -- (2) ;
\draw (2) -- (3) ;
\draw (3) -- (4) ;
\draw (4) -- (5) ;
\draw (5) -- (6) ;
\draw (6) -- (1) ;

\node[draw, circle] (7) at (3,2) [label=below:$d_{j,3,1}$]{};
\node[draw, circle] (8) at (4,2) [label=below right:$d_{j,3,5}$]{};
\node[draw, circle, fill = red] (9) at (4.5,3) [label=right:$d_{j,3,4}$]{};
\node[draw, circle, fill = ForestGreen] (10) at (3.5,3.7) [label=above:$d_{j,3,3}$]{};
\node[draw, circle, fill = red] (11) at (2.5,3) [label=left:$d_{j,3,2}$]{};

\draw (7) -- (8) ;
\draw (8) -- (9) ;
\draw (9) -- (10) ;
\draw (10) -- (11) ;
\draw (11) -- (7) ;

\draw (2) -- (7) ;

\node[draw, circle] (12) at (-3,2) [label=below:$d_{j,1,1}$]{};
\node[draw, circle, fill = red] (13) at (-4,2) [label=below left:$d_{j,1,2}$]{};
\node[draw, circle, fill = ForestGreen] (14) at (-4.5,3) [label=left:$d_{j,1,3}$]{};
\node[draw, circle, fill = red] (15) at (-3.5,3.7) [label=above:$d_{j,1,4}$]{};
\node[draw, circle] (16) at (-2.5,3) [label=right:$d_{j,1,5}$]{};

\draw (12) -- (13) ;
\draw (13) -- (14) ;
\draw (14) -- (15) ;
\draw (15) -- (16) ;
\draw (16) -- (12) ;

\draw (12) -- (6) ;

\node[draw, circle] (17) at (0,3.8) [label=right:$d_{j,2,1}$]{};
\node[draw, circle, fill = red] (18) at (-1,4.5) [label=left:$d_{j,2,2}$]{};
\node[draw, circle, fill = ForestGreen] (19) at (-0.5,5.5) [label=above left:$d_{j,2,3}$]{};
\node[draw, circle, fill = red] (20) at (0.5,5.5) [label=above right:$d_{j,2,4}$]{};
\node[draw, circle] (21) at (1,4.5) [label=right:$d_{j,2,5}$]{};

\draw (17) -- (18) ;
\draw (18) -- (19) ;
\draw (19) -- (20) ;
\draw (20) -- (21) ;
\draw (21) -- (17) ;

\draw (17) -- (1) ;

\end{tikzpicture}
  \caption{Graph $G_{C_j}$\label{f:npc-GCj}}
\end{figure}

Then, add all possible edges between the red vertices to $\{a,a',a'',a'''\}$, and all possible edges between the green vertices and $\{v,v'\}$.

For every $j=1, \dots, m$ and every $k=1, 2, 3$, let $i\in \{1, \dots,
n\}$ be such that $y_{j, k}=x_i$ or $y_{j, k} = \overline{x_i}$.  If
$y_{j, k}=x_i$, add the edge $f_id_{j, k, 5}$.  If
$y_{j, k}=\overline{x_i}$, add the edge $t_id_{j, k, 5}$.   This conclude
the construction of $G_{\mathcal I}$.

To conclude the proof of Theorem~\ref{th:npc}, it remains to prove
that $\mathcal I$ admits a truth assignment satisfying all clauses
if and only if $G_{\mathcal I}$ is a Zykov graph.

Suppose first that the variables of $\mathcal I$ admits a truth
assignment satisfying all clauses $\mathcal I$.  Let us build a stable
set $A$ of $G_{\mathcal I}$. First add to $A$ all green vertices, all vertices $t_i$ such that $x_i$ it true, and all
vertices $f_i$ such that $x_i$ is false.  For every clause $C_j$, choose an
integer $k_j\in \{1, 2, 3\}$ such that variable $y_{j, k} = x_i$ and
$x_i$ is true or $y_{j, k} = \overline{x_i}$ and $x_i$ is false (this
is possible since the clauses are all satisfied by the truth
assignment).  Then, add $c_{j, k_j}$ and $d_{j, k_j, 5}$ to $A$.  For
all $k'\in \{1, 2, 3\} \sm \{k_j\}$, add the vertex $d_{j, k', 1}$ to
$A$.  Observe that $A$ is stable set. In particular, there is no edge $d_{j, k, 5}f_i$ (resp. $d_{j, k, 5}t_i$) since $d_{j,k,5}$ is taken only when $C_j$ is satisfied by variable $x_i$ which appears positively (resp. negatively) in it, in which case $t_i$ (resp. $f_i$) is taken in $A$.
Also, observe that removing $A$ from $G_{\mathcal I}$ yields a forest whose components are isolated vertices or edges. Since $G_{\mathcal{I}}$ is triangle-free, it proves that $A$ is a splitting stable set. In addition, since forests are Zykov graph by Lemma~\ref{l:fz}, it holds that $G_{\mathcal{I}}$ is a Zykov graph by Condition~\eqref{i:es} of
Lemma~\ref{lemma:z}.

Conversely, suppose that $G_{\mathcal I}$ is a Zykov graph. By
Theorem~\ref{th:z}, $G_{\mathcal I}$ contains a non-empty splitting
stable set $A$.  We now prove several claims.

\begin{myclaim}
  \label{c:nored}
  $A$ contains no red vertex.
\end{myclaim}

\begin{proofclaim}
  By Lemma~\ref{l:h} no red vertex of $G_{\mathcal I}$ can be in $A$
  because every red vertex is contained in a copy of $H$.
\end{proofclaim}

\begin{myclaim}
  \label{c:egreen}
  $A$ contains at least one green vertex.
\end{myclaim}

\begin{proofclaim}
  Since $A$ is not empty, we may assume, because of \eqref{c:nored}, that $A$ contains an uncolored vertex. Let us check that it implies that a green vertex in also $A$.

  If $A$ contains $t_i$ or $f_i$ for some $i= 1, \dots, n$, then $A$
  must contain $b_{i, 2}$ by Lemma~\ref{l:2c} applied to $G_i$.

  If $A$ contains $d_{j,k, 1}$ or $d_{j, k, 5}$ for some
  $j=1, \dots, m$ and $k=1, 2, 3$, then it must contain $d_{j, k, 3}$
  by Lemma~\ref{l:2c} applied to
  $d_{j, k, 1}d_{j, k, 2}d_{j, k, 3}d_{j, k, 4}d_{j, k, 5}d_{j, k,
    1}$.

  If $A$ contains $c_{j, 2}$ for some $j= 1, \dots, m$, then it must contain
  $c_{j, 5}$ by Lemma~\ref{l:2c} applied to
  $c_{j, 1}c_{j,  2}c_{j, 3}c_{j, 4}c_{j, 5}c_{j, 6}c_{j, 1}$.

  If $A$ contains $c_{j,1}$ $j= 1, \dots, m$, then it must contain $a$
  by Lemma~\ref{l:2c} applied to
  $c_{j, 1}d_{j, 1, 1}d_{j, 1, 2}ac_{j, 6}c_{j, 1}$.  The proof is
  similar when $A$ contains $c_{j, 3}$.
\end{proofclaim}

\begin{myclaim}
  \label{c:agreen}
  $A$ contains all green vertices. 
\end{myclaim}

\begin{proofclaim}
  By~\eqref{c:egreen}, some green vertex is in $A$. So, by
  Lemma~\ref{l:h}, $a\in A$
  since every green vertex is contained together with $a$ in some
  copy of $H$.  Hence by Lemma~\ref{l:h} and the same remark, all
  green vertices are in $A$. 
\end{proofclaim}

By Lemma~\ref{l:2c} applied to $G_i$, we know that exactly one of
$t_i$ or $f_i$ is in $A$. If $t_i\in A$ we assign the value true to
$x_i$ and the value false otherwise.  We claim that this truth
assignment satisfies all clauses of $\mathcal I$.

Indeed, let $C_j$ be a clause. By Lemma~\ref{l:2c} applied to
$c_{j, 1}c_{j, 2}c_{j, 3}c_{j, 4}c_{j, 5}c_{j, 6}c_{j, 1}$, at least
one vertex among $c_{j, 1}$, $c_{j, 2}$ or $c_{j, 3}$ must be in $A$,
say $c_{j, k}$.  Suppose that $y_{j, k}= x_i$. If $x_i$ is assigned
value false, then $f_i\in A$ and
$f_id_{j, k, 5}\in E(G_{\mathcal I})$.  Hence, none of $d_{j,k,1}$ and
$d_{j, k, 5}$ is in $A$.  So, the cycle
$d_{j, k, 1}d_{j, k, 2}d_{j, k, 3}d_{j, k, 4}d_{j, k, 5}d_{j, k, 1}$
contradicts Lemma~\ref{l:2c}.  Hence, $x_i$ is assigned value true.
It follows that $C_j$ is satisfied.  The proof when
$y_{j, k}= \overline{x_i}$ is symmetric.  We proved that all clauses
are satisfied.  This concludes the proof of Theorem~\ref{th:npc}. 
\end{proof}

 A byproduct of our construction are the following.  We omit the proofs that are similar to the proof of Theorems~\ref{th:npc}. 

\begin{theorem}
  \label{th:sssNpc}
  It is NP-complete to decide whether an input graph contains a
  non-empty splitting stable set. 
\end{theorem}

\begin{theorem}
  \label{th:FS}
  It is NP-complete to decide whether an input graph can be
  (vertex-wise) partitioned into a stable set and a forest in such way
  that every vertex of the stable set has at most one neighbor in each
  component of the forest.  
\end{theorem}

\subsection{Augmenting the girth}

A graph $G=(V,E)$ is called a $(n,d,c)$-\emph{expander} if it has $n$ vertices, the maximum degree of a vertex is $d$, and 
$$
\min_{W\subseteq V, |W|<n/2} \frac{|N(W)|}{|W|} \geq  c
$$
where $N(W)$ denotes the set of vertices adjacent to $W$ in $V\setminus W$.

In addition, let $\lambda(G)$ be the second largest eigenvalue of the adjacency matrix of $G$ in absolute value. We state three well-known results on expanders.

\begin{theorem}[\protect{\cite[Theorem 9.2.1]{alon2015probabilistic}}]\label{thm:EdgeExpension}
For every partition of the set of vertices of an $(n,d, c)$-expander into two subsets $B$ and $C$ :
$$
e(B,C) \geq \frac{(d-\lambda)|B||C|}{n}
$$
where $e(B,C)$ is the number of edges between $B$ and $C$.
\end{theorem}

\begin{theorem}[\protect{\cite[Corollary 9.2.2]{alon2015probabilistic}}]\label{thm:Alon}
If $G$ is a $k$-regular graph with $n$ vertices, then $G$ is an $(n,k,c)$-expander for $c= \frac{k-\lambda(G)}{2k}$.
\end{theorem}

\begin{theorem}[\cite{lubotzky1988ramanujan}]\label{thm:ExpanderGirth}
For any $n\geq 1$, there exists a $k$-regular graph $G$ with at least $n$ vertices such that :
\begin{itemize}
\item $2\sqrt{k-1}-1 \leq \lambda(G) \leq 2\sqrt{k-1}$ ;
\item the girth of $G$ is at least $\frac{4}{3} \log_{k-1} n$.
\end{itemize}
\end{theorem}

\begin{theorem}\label{thm:ZykovLargeGirth}
There exists non-Zykov (and therefore non-Blanche-Descartes) graphs of arbitrarily large girth.
\end{theorem}

\begin{proof}
    Every Zykov graph is a Blanche Descarte graph as we will see in Lemma~\ref{l:ZiBD}, so we may focus on Zykov graphs. 
    Let $g \geq 4$ be an integer. By Theorem~\ref{thm:ExpanderGirth}, there exists a graph $G$ on $n$ vertices with girth at least $g$ and such that $\lambda =\lambda(G) \leq 2\sqrt{k-1}$, where $k$ is yet to be fixed. By Theorem~\ref{thm:Alon}, $G$ is an $(n,k,c)$-expander with $c= \frac{k-2\sqrt{k-1}}{2k}$.

    We fix $k$ large enough to satisfy the three following inequalities :
    \begin{enumerate}
        \item $k>\lambda$
        \item $1-\frac{\lambda}{k} \geq \frac{1}{2}$
        \item $\frac{4\lambda}{k - \lambda} \leq \frac{1}{2}$
    \end{enumerate}

    A well-known result about those graphs is the size of their stable sets. Indeed, any stable set $S\subseteq V(G)$ has size at most $\frac{\lambda}{k}n = \frac{2\sqrt{k-1}}{k}$~\cite{hoory2006expander}.

    By contradiction, assume that $G$ is Zykov, so by Theorem~\ref{th:z} it has a non-empty splitting stable set $S$ which has, by the previous remark, at most $\frac{\lambda}{k}n$ vertices.

    \begin{myclaim}
    There is no connected component $C$ in $G\setminus S$ such that $|C|>\frac{1}{k-\lambda}n$.
    \end{myclaim}
    
    \begin{proofclaim}
    Assume that there exists such a connected component $C$, and let $B = V(G)\setminus C$. By Theorem~\ref{thm:EdgeExpension}, 
    $$
    e(B,C) \geq \frac{k-\lambda}{n}|B||C|
    $$
    However, notice that the only edges between $B$ and $C$ are between $S$ and $C$, and there are at most $|S|$ of them since $S$ is splitting. Thus, $e(B,C) \leq |S|$. In addition, since $|B| \geq |S|$ and $|C|> \frac{1}{k-\lambda} n$, 
    $$
    |S| > \frac{k-\lambda}{n}|S| \frac{1}{k-\lambda}n
    $$
    which leads to a contradiction.
    \end{proofclaim}

    Let $A_1, \dots, A_\ell$ be the connected components of $G - S$, sorted in decreasing order of cardinality. By the previous claim, $|A_i| \leq \frac{1}{k - \lambda} n$ for all $1 \leq i \leq \ell$. Define $B_i = \bigcup_{1 \leq j \leq i} A_j$ for $1 \leq i \leq \ell$.
    Firstly, note that $|B_1| \leq \frac{1}{k - \lambda} n$. Secondly, note that $|B_\ell| = n-|S| \geq n-\frac{\lambda}{k}n$ and by the second inequality on $k$, we have $|B_\ell| \geq \frac{n}{2}$. Let $i \geq 2$ be the largest integer such that $|B_{i-1}| \leq \frac{2\lambda}{k - \lambda} n$, which is well-defined since $\lambda> 1$ from Theorem~\ref{thm:ExpanderGirth}. Then, adding $A_i$ to $B_{i-1}$ gives
    $$
    \frac{2\lambda}{k - \lambda} n < |B_i| \leq  \frac{2\lambda+1}{k - \lambda} n,
    $$
    since $|A_i| \leq \frac{1}{k - \lambda} n$.

    By the third inequality on $k$, we know $|B_i| \leq \frac{n}{2}$. Thus, $|N(B_i)| \geq c |B_i| > \frac{\lambda}{k} n$. However, since $N(B_i) \subseteq S$, we also have $|N(B_i)| \leq \frac{\lambda}{k} n$, which leads to a contradiction.
\end{proof}

Theorems~\ref{th:npc}, \ref{th:sssNpc} and~\ref{th:FS} may still be true under the
assumption that the input graph has high girth, but we leave this open.

\section{Blanche Descartes construction}\label{sec:BD}

Similarly to Zykov graphs, we derive a structural characterization of Blanche Descartes graphs, also based on special stable sets. In particular, as a corollary of this characterization, we show that all Blanche Descartes graphs are Zykov graphs. Then, we prove that recognizing Blanche Descartes graphs is also NP-complete. Finally, we show that both \textsc{Maximum Independent Set} and k-\textsc{Coloring} remain NP-complete on Blanche-Descartes, and therefore on Zykov graphs as well.

\subsection{Characterization and properties of Blanche Descartes graphs}
A stable set $S$ in some graph $G$ is \emph{strongly splitting} if every vertex of $S$ has at most one neighbor in each connected component of
$G\sm S$ and each vertex of $G-S$ has at most one neighbor in $S$. In the same way as splitting sets, a graph contains a strongly splitting stable set if and only if at least one of its connected component does.
\begin{lemma}
  \label{lemma:BD}
  For all graphs $G$, the following conditions are equivalent:

  \begin{enumerate}
  \item\label{i:BD} $G$ is a Blanche Descartes graph.
   \item\label{i:asBD} Every induced subgraph $H$ of $G$ contains a
    non-empty strong splitting stable set $S$ such that $H \setminus S$ is a Blanche Descartes
    graph.
  \item\label{i:esBD} $G$ contains a non-empty strong splitting stable set $S$ such that
    $G \setminus S$ is a Blanche Descartes graph.
  \end{enumerate}
\end{lemma}

\begin{proof}
    To prove that~\eqref{i:BD} implies~\eqref{i:asBD}, it is enough to prove that every connected induced subgraph $H$ of some Blanche Descartes graph $G$ contains a non-empty strong splitting stable set. Indeed, if $G$ is a Blanche Descartes graph and $H$ and non connected induced subgraph of $G$, thus if each connected component  of $H$ has a non-empty strong splitting set, so has $H$.  So, let $H$ be a connected subgraph of some graph $D_k$ and suppose that $k$ is minimal with respect to this property. By the minimality of $k$ and connectivity of $H$, $H$ contains some vertices of the set $S_k$ used to construct $D_k$ from  the copies of $D_{k-1}$ as in the definition. Hence, $V(H)\cap S_k$ is a non-empty splitting stable set of $H$, so Condition~\eqref{i:asBD} holds.
  
  Clearly, \eqref{i:asBD} implies~\eqref{i:esBD}.

  To prove that~\eqref{i:esBD} implies~\eqref{i:BD}, suppose that $G$ has a non-empty strong splitting stable set $S$ such that $G-S$ is a Blanche Descartes graph. Let $C_1, \ldots, C_\ell$ be the connected components of $G - S$, and since $G-S$ is a Blanche Descartes graph, each of them is an induced subgraph of some Blanche Descartes graph, say $D_{k_1}, \ldots, D_{k_\ell}$ respectively. Let $k$ be an integer large enough such that $k \geq \max_{1 \leq i \leq \ell} k_i$ and $|V(D_k)| > |S|$, and let $n=|V(D_k)|$. We claim that $G$ is isomorphic to an induced subgraph of $D_{k+\ell+1}$. Consider the connected components $C_1, \ldots, C_\ell$ as induced subgraphs of $\ell$ disjoint copies of $D_{k+\ell}$. Then, for each $i =1,...,\ell$, let $S_i$ be the set of vertices of $S$ adjacent to some vertex of $C_{i}$, and add $n-n_i$ new vertices $\{c_{i,1},...,c_{i,n-n_i}\} := S_i'$, where $n_i = |S_i|$ to $G$, and add a matching between them and $S_i$. Notice that $S\cup S_1'\cup ... \cup S_\ell'$ has size at most $|S| + n\ell \leq k+n\ell$. Thus, add enough vertices to complete $S\cup S_1'\cup ... \cup S_\ell'$ into a stable set $S'$ of size $(k+\ell)n+1$. Finally, for each $n$-tuple which is not matched with a copy of $D_{k+\ell}$, add such a copy and a matching to obtain $D_{k+\ell+1}$.

\end{proof}

\begin{theorem}
  \label{th:BD}
  A graph $G$ is a Blanche Descartes graph if and only if all induced subgraphs
  of $G$ contain a non-empty strong splitting stable set.
\end{theorem}

\begin{proof}
  If $G$ is a Blanche Descartes graphs, the conclusion holds by
  Condition~\eqref{i:asBD} of Lemma~\ref{lemma:BD}. Conversely, suppose that
  $G$ is such that all induced subgraphs of $G$ contain a non-empty
  strong splitting stable set.  Let us prove by induction on $|V(G)|$ that
  $G$ is a Blanche Descartes graph. If $|V(G)|=1$, $G=D_1$ is a Blanche Descartes graph. If
  $|V(G)|>1$, then by assumption $G$ contains a non-empty strong splitting
  stable set $A$.  By the induction hypothesis, $G\sm A$ is a Blanche Descartes
  graph.  Hence, by Condition~\eqref{i:esBD} of Lemma~\ref{lemma:BD}, $G$ is
  a Blanche Descartes graph.
\end{proof}

\begin{lemma}
\label{l:ZiBD}
Every Blanche Descartes graph is a Zykov graph.
\end{lemma}

\begin{proof}
A strong splitting stable set is also by definition a splitting stable set. From Theorem~\ref{th:z} and Theorem~\ref{th:BD}, the result holds.
\end{proof}

Using a similar proof technique as in Lemma \ref{l:allSub}, it follows that Blanche Descartes graphs are closed under edge subdivision and taking subgraph.
\begin{lemma}
\label{l:BDallSub}
The class of Blanche Descartes graphs is closed under subdividing edges and under taking subgraphs.
\end{lemma}

We now give some tools to prove that some graph are not Blanche Descartes graphs. We omit the proofs which are very similar to their equivalent with Zykov graphs.

\begin{lemma}
\label{l:BDd3}
If $S$ is a strong splitting set of some graph $G$, then there is no vertex $v$ in $G$ adjacent to two distinct vertices of $S$.
\end{lemma}

\begin{lemma}
\label{l:BD2c}
If $C$ is a cycle of some graph $G$, then every strong splitting stable set $S$ of $G$ contains either no vertex of $C$, or at least two vertices of $C$ at distance at least $3$ in $C$.
\end{lemma}

\begin{proof}
Let $S$ be a strong splitting stable set of $G$. By Lemma \ref{l:2c}, $S$ contains either no vertex of $C$, or at least two of them. If $S$ contains two of them at distance $2$, then there exists a vertex of $G-S$ adjacent to two vertices of $S$, which contradict the definition of strongly splitting stable sets.
\end{proof}

\begin{figure}
\centering
\begin{tikzpicture}

\node[draw,circle, fill =ForestGreen] (0) at (0,0) [label=left:$a$]{};
\node[draw,circle, fill =red] (1) at (1,0) {};
\node[draw,circle, fill =red] (2) at (2,0) {};
\node[draw,circle, fill =ForestGreen] (3) at (3,0) [label=right:$a'$]{};

\node[draw,circle, fill =red] (4) at (1,1) {};
\node[draw,circle, fill =red] (5) at (2,1) {};
\node[draw,circle, fill =red] (6) at (1,-1) {};
\node[draw,circle, fill =red] (7) at (2,-1) {};

\draw (0)--(1)--(2)--(3);
\draw (0)--(4)--(5)--(3);
\draw (0)--(6)--(7)--(3);

\end{tikzpicture}

  \caption{Graph $L$\label{f:npc-L}}
\end{figure}

\begin{lemma}
\label{l:BDL}
If a graph $G$ contains the graph $L$ represented on Figure \ref{f:npc-L}, then for every strongly splitting stable set $S$ of $G$, $S\cap V(L)=\emptyset$ or $S\cap V(L) = \{a,a'\}$.
\end{lemma}

Notice that any vertex of degree at most $1$ is also a strong splitting stable set. Since any forest has such a vertex, they are all Blanche Descartes graphs.
\begin{lemma}
\label{l:BDf}
All forests are Blanche Descartes graphs.
\end{lemma}

\begin{lemma}
  \label{l:allSubBD}
  Let $G$ be any graph (possibly not Blanche Descartes). If $H$ is obtained by
  subdividing all edges of $G$ at least twice, then $H$ is a Blanche Descartes
  graph.
\end{lemma}

\begin{proof}
  If each edge of $G$ is subdivided exactly twice, then $H$ is
  Blanche Descartes because original vertices of $G$  form a strong splitting stable set whose removal yields a forest of edges, so the result follows from Lemmas~\ref{lemma:BD} and~\ref{l:BDf}. If there
  are more subdivisions, then the result follows from Lemma~\ref{l:BDallSub}. 
\end{proof}

Observe that subdividing twice in Lemma~\ref{l:allSubBD} is best possible, since for instance subdividing once each edge of $K_4$ yields a non-Blanche-Descartes graph.  Similarly to Zykov graphs, Theorem~\ref{th:BD} can be used to describe Blanche Descartes using a monadic second order formula. Hence the following holds.

\begin{theorem}\label{thm:BDFPT}
    There exists an algorithm that, given an input graph $G$, decides if $G$ is a Blanche Descartes graph in time $f(\text{tw}(G)) \cdot n$, where $n$ is the number of vertices of $G$, $tw(G)$ its treewidth and $f$ a computable function.
\end{theorem}

\subsection{NP-hardness}\label{sec:BDnpc}

\begin{theorem}\label{th:BDnpc}
    Recognizing Blanche Descartes graphs is NP-complete.
\end{theorem}

\begin{proof}
The notations and the method of proof are similar to the proof of Theorem~\ref{th:npc} but the details differ. We  reduce 3-{\sc sat} to the problem of recognizing Blanche Descartes graphs.
Consider an instance $\mathcal I$ of 3-{\sc sat} made of $n$ variables
$x_1$, \dots, $x_n$ and $m$ clauses $C_1, \dots, C_m$ each on three
variables.  For each $j=1, \dots, m$, 
$C_j=y_{j, 1} \vee y_{j, 2} \vee y_{j, 3}$ where for all $k=1, 2, 3$,
there exists $1\leq i \leq n$ such that $y_{j, k} = x_i$ or
$y_{j, k} = \overline{x_i}$.

We define a graph $G_{\mathcal I}$ depending on $\mathcal I$.  To make
the explanations easier, a color (either red or green) is assigned to
some vertices of $G_{\mathcal I}$.  The red vertices will turn out to
be in no strong splitting stable set of $G_{\mathcal I}$ and the green
vertices in all non-empty strong splitting stable set of $G_{\mathcal
  I}$. The presence of uncolored vertices in some non-empty strong splitting
stable set will depend on $\mathcal I$.

Prepare a copy of the gadget $L_{k,\ell}$. It is obtained by taking $\ell+k$ copies of the graph $L$, identifying one green vertex per copy into a common vertex called $\alpha$. The other green vertices are called $\alpha_1,...,\alpha_\ell$ and $\gamma_1,...,\gamma_k$. Then, add $k$ red vertices $\beta_1,...,\beta_k$ and for all $i\in \{1,...,k\}$ add a red vertex adjacent to both $\gamma_i$ and $\beta_i$. $L_{k,\ell}$ is depicted on Figure~\ref{f:npc-BDLkl}, where $k$ and $\ell$ are yet to be fixed, and will be exactly the number of red and green vertices respectively in the rest of the graph. The vertices $\alpha_1,...,\alpha_\ell$ (resp.\ $\beta_1,...,\beta_k$) are called the green ports (resp.\ red ports). The other vertices are the \textit{private} vertices of $L_{k,\ell}$.

For each variable $x_i$, prepare a graph $G_{i}$ on 7 vertices
$t_i$, $f_i$, $b_{i, k}$ for $k\in \{1,...,5\}$ (see Figure~\ref{f:npc-BD-Gi}).  Add edges in
such way that $t_ib_{i,1}...b_{i,5}f_it_i$ is a cycle. Give color
green to vertex $b_{i, 3}$ and give color red to the other vertices except $t_i$ and $f_i$, which have no color.

\begin{figure}
\centering
    \begin{tikzpicture}
    \node[draw, circle, fill = ForestGreen] (0) at (0,0) [label=left: $\alpha$]{};

    \node[draw, circle, fill = ForestGreen] (2) at (-3,1) [label = below:$\alpha_1$]{};

    \node[draw, circle, fill = ForestGreen] (3) at (-3,-1) [label = below:$\alpha_\ell$]{};
    \foreach \i in {-1,0,1}{
        \node[draw, circle, fill = red] (\i-3) at (-1,1+\i*0.5){};
        \node[draw, circle, fill = red] (\i-4) at (-2,1+\i*0.5){};
        \draw (0) -- (\i-3) -- (\i-4) -- (2) ;
        
        \node[draw, circle, fill = red] (\i-5) at (-1, \i*0.5-1){};
        \node[draw, circle, fill = red] (\i-6) at (-2, \i*0.5-1){};
        \draw (0) -- (\i-5) -- (\i-6) -- (3) ;  
    }

    \node[draw, circle, fill = ForestGreen] (4) at (3,1) [label = below:$\gamma_1$]{};

    \node[draw, circle, fill = ForestGreen] (5) at (3,-1) [label = below:$\gamma_k$]{};
    \foreach \i in {-1,0,1}{
        \node[draw, circle, fill = red] (\i-3) at (1,1+\i*0.5){};
        \node[draw, circle, fill = red] (\i-4) at (2,1+\i*0.5){};
        \draw (0) -- (\i-3) -- (\i-4) -- (4) ;
        
        \node[draw, circle, fill = red] (\i-5) at (1, \i*0.5-1){};
        \node[draw, circle, fill = red] (\i-6) at (2, \i*0.5-1){};
        \draw (0) -- (\i-5) -- (\i-6) -- (5) ;  
    }

    \node[draw, circle, fill = red] (6) at (4,1) {};
    \node[draw, circle, fill = red] (7) at (5,1) [label=right:$\beta_1$]{};
    \draw (4)--(6)--(7);

    \node[draw, circle, fill = red] (6bis) at (4,-1) {};
    \node[draw, circle, fill = red] (7bis) at (5,-1) [label=right:$\beta_k$]{};
    \draw (5)--(6bis)--(7bis);

    \end{tikzpicture}
\caption{The gadget $L_{k,\ell}$. }\label{f:npc-BDLkl}
\end{figure}
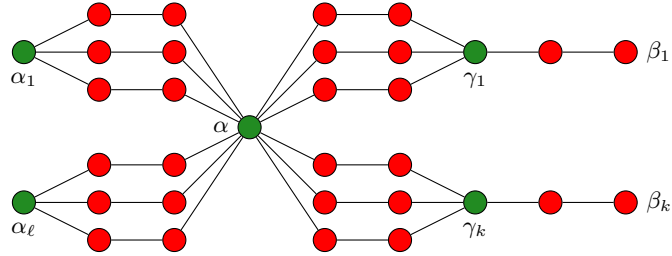

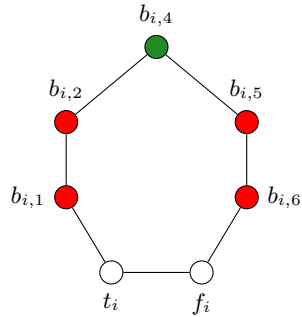
\begin{figure}
\centering 
  
\begin{tikzpicture}

\node[draw, circle, fill = ForestGreen] (1) at (0,3) [label=above:$b_{i,4}$]{};
\node[draw, circle, fill = red] (2) at (1.2,2) [label=above:$b_{i,5}$]{};
\node[draw, circle, fill = red] (3) at (1.2,1) [label=right:$b_{i,6}$]{};

\node[draw, circle] (4) at (0.6,0) [label=below:$f_i$]{};
\node[draw, circle] (5) at (-0.6,0) [label=below:$t_i$]{};
\node[draw, circle, fill = red] (6) at (-1.2,1) [label=left:$b_{i,1}$]{};
\node[draw, circle, fill = red] (7) at (-1.2,2) [label=above:$b_{i,2}$]{};

\draw (1) -- (2)--(3)--(4)--(5)--(6)--(7)--(1);
\end{tikzpicture}
\caption{Graph $G_i$}\label{f:npc-BD-Gi} 
\end{figure}

For each clause $C_j$, prepare a graph $G_{C_j}'$ on 35 vertices with 4 green vertices and 22 red vertices as in Figure~\ref{f:npc-GCj'}. 

For every $j=1, \dots, m$ and every $k=1, 2, 3$, let $i\in \{1, \dots,
n\}$ be such that $y_{j, k}=x_i$ or $y_{j, k} = \overline{x_i}$.  If
$y_{j, k}=x_i$ (resp. $y_{j,k}=\overline{x_i}$), add the edge $f_ie_{j,k}$ (resp. $t_id_{j, k, 7}$). 

Let us now fix the values of $k$ and $\ell$ of $L_{k, \ell}$, which correspond to the number of green and red vertices constructed so far, respectively (apart from those of $L_{k, \ell}$). That is, $k=4m+n$ and $\ell=22m+4n$. Now, pick an arbitrary bijection between the green (resp. red) vertices and $\alpha_1$, $\dots$, $\alpha_\ell$ (resp. $\beta_1$, $\dots$, $\beta_{\ell}$) of $L_{k,\ell}$, and identify the two mapped vertices.

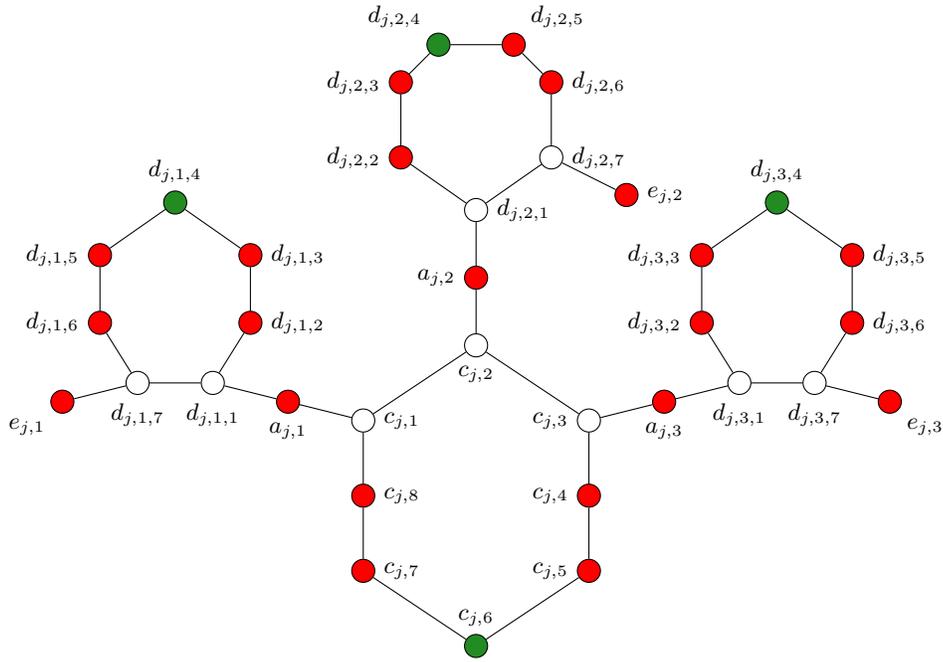
\begin{figure}
\centering

\begin{tikzpicture}
\node[draw, circle] (1) at (0,2.5) [label=below:$c_{j,2}$]{};
\node[draw, circle] (2) at (1.5,1.5) [label=left:$c_{j,3}$]{};
\node[draw, circle, fill = red] (3) at (1.5,0.5) [label=left:$c_{j,4}$]{};
\node[draw, circle, fill = red] (3bis) at (1.5,-0.5) [label=left:$c_{j,5}$]{};
\node[draw, circle, fill = ForestGreen] (4) at (0,-1.5) [label=above:$c_{j,6}$]{};
\node[draw, circle, fill = red] (5) at (-1.5,-0.5) [label=right:$c_{j,7}$]{};
\node[draw, circle, fill = red] (5bis) at (-1.5,0.5) [label=right:$c_{j,8}$]{};
\node[draw, circle] (6) at (-1.5,1.5) [label=right:$c_{j,1}$]{};

\draw (1) -- (2) -- (3) -- (3bis) -- (4) -- (5) -- (5bis) -- (6) -- (1);

\node[draw, circle] (7) at (3.5,2) [label=below:$d_{j,3,1}$]{};
\node[draw, circle] (8) at (4.5,2) [label=below:$d_{j,3,7}$]{};
\node[draw, circle, fill = red] (9) at (5,2.8) [label=right:$d_{j,3,6}$]{};
\node[draw, circle, fill = red] (9bis) at (5,3.7) [label=right:$d_{j,3,5}$]{};
\node[draw, circle, fill = ForestGreen] (10) at (4,4.4) [label=above:$d_{j,3,4}$]{};
\node[draw, circle, fill = red] (10bis) at (3,3.7) [label=left:$d_{j,3,3}$]{};
\node[draw, circle, fill = red] (11) at (3,2.8) [label=left:$d_{j,3,2}$]{};

\draw (7) -- (8) --(9)--(9bis)--(10)--(10bis)--(11)-- (7);

\node[draw, circle, fill = red] (2bis) at (2.5,1.75) [label = below:$a_{j,3}$]{};
\node[draw,circle, fill = red] (8bis) at (5.5,1.75) [label=below right:$e_{j,3}$]{};
\draw (8) -- (8bis) ;
\draw (2) -- (2bis) --  (7) ;

\node[draw, circle] (7) at (-3.5,2) [label=below:$d_{j,1,1}$]{};
\node[draw, circle] (8) at (-4.5,2) [label=below:$d_{j,1,7}$]{};
\node[draw, circle, fill = red] (9) at (-5,2.8) [label=left:$d_{j,1,6}$]{};
\node[draw, circle, fill = red] (9bis) at (-5,3.7) [label=left:$d_{j,1,5}$]{};
\node[draw, circle, fill = ForestGreen] (10) at (-4,4.4) [label=above:$d_{j,1,4}$]{};
\node[draw, circle, fill = red] (10bis) at (-3,3.7) [label=right:$d_{j,1,3}$]{};
\node[draw, circle, fill = red] (11) at (-3,2.8) [label=right:$d_{j,1,2}$]{};

\draw (7) -- (8) --(9)--(9bis)--(10)--(10bis)--(11)-- (7);

\node[draw, circle, fill = red] (6bis) at (-2.5,1.75) [label = below:$a_{j,1}$]{};
\node[draw, circle, fill = red] (8bis) at (-5.5,1.75) [label = below left:$e_{j,1}$]{};
\draw (8) -- (8bis) ;
\draw (7) -- (6bis) -- (6) ;

\node[draw, circle] (17) at (0,4.3) [label=right:$d_{j,2,1}$]{};
\node[draw, circle, fill = red] (18) at (-1,5) [label=left:$d_{j,2,2}$]{};
\node[draw, circle, fill = red] (18bis) at (-1,6) [label=left:$d_{j,2,3}$]{};

\node[draw, circle, fill = ForestGreen] (19) at (-0.5,6.5) [label=above left:$d_{j,2,4}$]{};
\node[draw, circle, fill = red] (19bis) at (0.5,6.5) [label=above right:$d_{j,2,5}$]{};
\node[draw, circle, fill = red] (20) at (1,6) [label=right:$d_{j,2,6}$]{};
\node[draw, circle] (21) at (1,5) [label=right:$d_{j,2,7}$]{};

\draw (17) -- (18) --(18bis) -- (19) -- (19bis) -- (20) -- (21) -- (17);

\node[draw, circle, fill = red] (1bis) at (0,3.4) [label = left:$a_{j,2}$]{};
\node[draw, circle, fill = red] (21bis) at (2,4.5) [label = right:$e_{j,2}$]{};

\draw (17) -- (1bis)-- (1) ;
\draw (21) -- (21bis) ;

\end{tikzpicture}
  \caption{Graph $G_{C_j}$\label{f:npc-GCj'}}
\end{figure}

To conclude the proof of Theorem~\ref{th:BDnpc}, it remains to prove
that $\mathcal I$ admits a truth assignment satisfying all clauses
if and only if $G_{\mathcal I}$ is a Blanche Descartes graph.

Suppose first that the variables of $\mathcal I$ admit a truth
assignment satisfying all clauses of $\mathcal I$.  Let us build a stable
set $S$ of $G_{\mathcal I}$. First add to $S$ all green vertices. Add also
to $S$ all vertices $t_i$ (resp. $f_i$) such that $x_i$ it true (false). For every clause $C_j$, choose an integer $k_j\in \{1, 2, 3\}$ such that $y_{j, k} = x_i$ and
$x_i$ is true or $y_{j, k} = \overline{x_i}$ and $x_i$ is false (this
is possible since the clauses are all satisfied by the truth
assignment).  Then, add $c_{j, k_j}$ and $d_{j, k_j, 7}$ to $S$.  For
all $k'\in \{1, 2, 3\} \sm \{k_j\}$, add the vertex $d_{j, k', 1}$ to
$S$. Observe that $S$ is a stable set. We now prove it is a strong splitting stable set.

\begin{myclaim}
    $G_\mathcal{I}-S$ is a forest.
\end{myclaim}

\begin{proofclaim}
    It is sufficient to show that any cycle of $G_\mathcal{I}$ contains a vertex of $S$. If a cycle contains a private vertex of $L_{k,\ell}$, then it must contain $\alpha$, and thus contains a green vertex which is in $S$. In addition, notice that any cycle contained entirely in a clause gadget or a variable gadget go through a green vertex which is also in $S$. Thus, if there is a cycle in $G_\mathcal{I}-S$, we can extract a path from a variable gadget to a clause gadget and to another distinct variable gadget. Such a path must contain some vertices $c_{j,k}, d_{j,k,1},d_{j,k,7}$ for some $j \in \{1,...,m\}$ and $k\in \{1,2,3\}$, and at least one of them is in $S$.
\end{proofclaim}

\begin{myclaim}
    $S$ is a strong splitting stable set.
\end{myclaim}

\begin{proofclaim}
    First, by construction, all vertices of $S$ are at distance at least $3$ from each other. Indeed, if $d_{j,k,7}\in S$ for some $k\in \{1,2,3\}$ and $j\in \{1,...,m\}$, then if $y_{k,j}=x_i$ for some $i\in \{1,...,n\}$, then $t_i\in S$ and $e_{j,3}t_i\notin E(G_\mathcal{I})$. The case $y_{j,k}=\overline{x_i}$ is similar, and all the other cases are trivial by construction. From that remark, it is sufficient to show that $S$ is a splitting stable set. 

    By contradiction, assume that a vertex $v\in S$ is adjacent to two distinct vertices in the same connected component of $G_\mathcal{I}-S$. First, observe that $v$ cannot be green :
    \begin{itemize}
        \item all the neighbors of $\alpha$ lie in different connected components in $G_\mathcal{I}-S$ since they are in a unique path from $\alpha$ to another green vertex ;
        \item the same remark holds for $\gamma_i$'s
        \item any $\alpha_i$ has two kinds of neighbors : the private vertices of $L_{k,\ell}$ which all lie in different connected components in $G_{\mathcal{I}}-S$ and its two red neighbors in either a clause gadget or a vertex gadget. In all cases, they are not in the same connected component since another vertex from the cycle is in $S$, and there is no path between them whose vertices are in $V(L_{k,\ell}) \setminus S$. 
    \end{itemize}
    The same reasoning than the one for $\alpha_i$'s works for uncoloured vertices of $S$.
\end{proofclaim}

Conversely, suppose that $G_{\mathcal I}$ is a Blanche Descartes graph. By Theorem~\ref{th:BD}, $G_{\mathcal I}$ contains a non-empty strong splitting stable set $S$.  We now prove several claims.

\begin{myclaim}
  \label{c:BDnored}
  $S$ contains no red vertex.
\end{myclaim}

\begin{proofclaim}
  By contradiction, assume that $S$ contains a red vertex $\beta_i$ for some $i\in \{1,...,k\}$. Then $\gamma_i\notin S$ since both vertices are at distance $2$ from each other. By Lemma~\ref{l:BDL}, $\alpha$ is not in $S$ since there is a copy of $L$ between $\alpha$ and $\gamma_i$. Notice that, by construction, each red vertex $\beta_i$ is adjacent to another red vertex $\beta_j$ with $j\neq i$. Notice that $\gamma_j\notin S$ since $\alpha\notin S$ by Lemma~\ref{l:BDL}, and the red vertex between $\gamma_j$ and $\beta_j$ cannot be in $S$ since it is at distance $2$ from $\beta_i$. Thus, $\beta_i\notin S$.

  We a similar argument we show that the same result holds for all red vertices exclusive to $L_{k,\ell}$.
\end{proofclaim}

\begin{myclaim}
  \label{c:BDegreen}
  $S$ contains at least one green vertex.
\end{myclaim}

\begin{proofclaim}
  Since $S$ is not empty and by \eqref{c:BDnored}, we may assume that
  $S$ contains an uncolored vertex. Let us check that the presence of
  any uncolored vertex in $A$ entails a green vertex in $A$.

  If $S$ contains $t_i$ or $f_i$ for some $i= 1, \dots, n$, then $S$
  must contain $b_{i, 4}$ because $G_i$ induces a cycle.

  If $S$ contains $d_{j,k, 1}$ or $d_{j, k, 7}$ for some
  $j=1, \dots, m$ and $k=1, 2, 3$, then it must contain $d_{j, k, 4}$
  since they are in a cycle.

  If $S$ contains $c_{j, k}$ for some $j= 1, \dots, m$ and $k=1, 2, 3$, then it must contain $c_{j, 6}$ since they are in a cycle and the two other uncolored vertices are at distance at most $2$.

  If $S$ contains $c_{j,1}$ $j\in \{ 1, \dots, m\}$, then it must either $c_{j,6}$ or $c_{j,3}$. However, if it contains $c_{j,3}$, then $c_{j,2}$ is adjacent to two vertices of $S$ and thus $S$ is not strongly splitting. The proof is similar when $S$ contains $c_{j, 3}$.
\end{proofclaim}

\begin{myclaim}
  \label{c:BDagreen}
  $S$ contains all green vertices. 
\end{myclaim}

\begin{proofclaim}
  By~\eqref{c:BDegreen}, some green vertex is in $S$. So, by
  the property of $L$, $\alpha\in S$
  since every green vertex is contained together with $\alpha$ in some
  copy of $L$.  Hence, all
  green vertices are in $S$. 
\end{proofclaim}

Since $G_i$ is a cycle, we know that exactly one of
$t_i$ or $f_i$ is in $A$. If $t_i\in S$ we assign the value true to
$x_i$ and the value false otherwise.  We claim that this truth
assignment satisfies all clauses of $\mathcal I$.

Indeed, let $C_j$ be a clause. Since
$c_{j, 1}\cdots c_{j, 6}c_{j, 1}$ is a cycle, at least
one vertex among $c_{j, 1}$, $c_{j, 2}$ or $c_{j, 3}$ must be in $A$,
say $c_{j, k}$.  Suppose that $y_{j, k}= x_i$. If $x_i$ is assigned
value false, then $f_i\in S$ and
$f_ie_{j,k}, e_{j,k}d_{j, k, 7}\in E(G_{\mathcal I})$.  Hence, none of $d_{j,k,1}$ and
$d_{j, k, 7}$ is in $S$ since otherwise they would share a common neighbor with a vertex of $S$, which is a contradiction.  Hence, $x_i$ is assigned value true.
It follows that $C_j$ is satisfied.  The proof when
$y_{j, k}= \overline{x_i}$ is symmetric.  We proved that all clauses
are satisfied.  This concludes the proof of Theorem~\ref{th:BDnpc}. 
\end{proof}

\subsection{NP-hard problems on Blanche Descartes graphs}

Given a graph $G$ and an integer $k$, the \textsc{Maximum Independent Set} problem asks whether $G$ contains a stable set of size at least $k$, and the k-\textsc{Coloring} problem asks whether $G$ has chromatic number at most $k$. Until the very recent work of Rz{\k{a}}{\.z}ewski and Walczak in which \textsc{Maximum Independent Set} was proved to be polynomially tractable in Burling graphs, all the hereditary graph classes on which those problems are tractable were also $\chi$-bounded. However, it remains an open question whether any graph class in which the k-\textsc{Coloring} problem is in P is necessarily $\chi$-bounded. 

Here, we show that both problems remain NP-complete on Blanche Descartes graphs, and, as a corollary, on Zykov graphs as well by Lemma~\ref{l:ZiBD}.

\begin{theorem}
    \label{thm:BDMIS}
    \textsc{Maximum Independent Set} is NP-complete in Blanche Descartes Graphs.
\end{theorem}

\begin{proof}
    We reduce from \textsc{Maximum Independent Set} on arbitrary graphs. From an instance $(G,k)$, we construct the graph $G^+$ from $G$ by subdividing four times each edge, meaning that we replace each edge $uv\in E(G)$ by a path $uw_1w_2w_3w_4v$ in $G^+$. It follows from \cite{poljak:74} that $G$ contains an independent set of size $k$ if and only if $G^+$ contains an independent set of size $k+2m$, where $m=|E(G)|$. In addition, $G^+$ is a Blanche Descartes graph by Lemma~\ref{l:allSubBD}.
\end{proof}

\begin{theorem}
    \label{thm:BDColoring}
    \textsc{3-Coloring} is NP-complete in Blanche Descartes Graphs.
\end{theorem}

\begin{proof}
Before the reduction, let us construct a Blanche Descartes graph $H$ with two special vertices $a$ and $b$, such that $\chi(H)=3$ and in any $3$-coloring of $H$, $a$ and $b$ have different colors. The construction is as follows:
 \begin{itemize}
        \item Take a stable set $S =\{s_1,\ldots, s_{19}\}$ on $19$ vertices.
        \item For each subset $T$ of $S$ of $7$ vertices, add a copy of $C_7$, the cycle with $7$ vertices, and a matching between $T$ and $C_7$.
        \item Remove the edge between $s_1$ and the copy of $C_7$ associated with $\{s_1,\cdots, s_7\}$.
    \end{itemize}

    We first prove that $H$ has a proper $3$-coloring. Start by coloring the vertices of $S$ such that the only monochromatic subset of size $7$ is $\{s_1,\cdots, s_7\}$. It is possible for instance by giving color $2$ to $s_1,\cdots, s_7$, and then color $i\mod 2$ to $s_i$ for $8\leq i \leq 19$. Then, notice that the copy of $C_7$ associated with $\{s_1,...,s_7\}$ can be properly colored using the color $2$ for the only vertex which was connected to $s_1$ before the edge was removed, and completing the coloring using colors $0$ and $1$. All the others $C_7$s can be colored using $3$ colors since the subsets of $S$ matched to them are not monochromatic.

    Then, we prove that $s_1$ and $s_8$ have different colors in any $3$-coloring of $H$. Note that in any $3$-coloring of $H$, one subset of $7$ vertices of $S$ has to be monochromatic, by the pigeonhole principle. This subset must be $\{s_1,...,s_7\}$, otherwise there is perfect matching between a $C_7$, which has no $2$-coloring, and a monochromatic set of $7$ vertices. Then, $s_8$ has to have a different color from $s_1$, otherwise $\{s_2,...,s_8\}$ would also be monochromatic.

    We are now ready for the reduction, which is from \textsc{$3$-Coloring} on arbitrary graphs. Let $G$ be an instance of this problem. 
    We construct a new graph $G'$ as follows : start from a copy of $G$, and for each $uv\in E(G)$, remove this edge, create one copy of $H$ and identify $s_1$ and $s_8$ with respectively $u$ and $v$.
    
    A direct consequence of the properties of $H$ is that $G$ has a proper $3$-coloring if and only if $G'$ has one. In addition, $G'$ is a Blanche Descartes graph. Indeed, let $I$ be the union of all stable sets $S$ of all copies of $H$ in $G'$. Notice that $G'-I$ is a disjoint union of $C_7$, which is a Blanche Descartes graph, and each vertex from $I$ is adjacent to at most one vertex from each $C_7$ and reciprocally. 
\end{proof}

\begin{credits}
\subsubsection{\ackname} The authors  are partially supported by the French National Research Agency under research grant ANR DIGRAPHS ANR-19-CE48-0013-01, ANR Twin-width ANR-21-CE48-0014-01 and the LABEX MILYON  (ANR-10-LABX-0070) of Université de Lyon, within the program Investissements d’Avenir (ANR-11-IDEX-0007) operated by the French National Research Agency (ANR).

\end{credits}

\bibliographystyle{plain}
\bibliography{biblio.bib}






\end{document}